\documentclass[11pt]{article}
\usepackage{amsmath}
\usepackage{mathrsfs}
\usepackage{amsmath}
\usepackage{amssymb}
\usepackage{amsfonts,amsthm,amssymb}
\usepackage{amsfonts}
\usepackage{graphics}
\usepackage{color}
\oddsidemargin\topmargin=-1.5cm

\usepackage{amssymb,latexsym,bm}
\newtheorem{lem}{Lemma}[section]
\newtheorem{Theo}{Theorem}[section]
\newtheorem{Coro}{Corollary}[section]
\newtheorem{Prop}{Proposition}[section]
\newtheorem{rem}{Remark}[section]

\newtheorem{thm}[Theo]{Theorem}

\newtheorem{exm}{Example}[section]

\theoremstyle{definition} \rm

\usepackage{amsfonts}
\usepackage{mathrsfs}
\usepackage{amsmath}
\usepackage{amsfonts,amsthm,amssymb,mathrsfs}
\usepackage{graphics}
\usepackage{color}
\usepackage{cite}
\usepackage{latexsym,bm}
\usepackage{fancyhdr}
\usepackage{amsmath,amsfonts,amssymb,graphicx,amsthm}    
\usepackage{subfigure}   
\usepackage{indentfirst} 
\usepackage{bm}          
\usepackage{multicol}    
\usepackage{indentfirst} 
\usepackage{picins}      
\usepackage{abstract}  

\begin{document}

\title{\bf The existence of the graphs that have exactly two main eigenvalues}
\author{
{\ Lin Chen }
\\[2mm]
\footnotesize College of Medical Engineering and Technology,
\footnotesize Xinjiang Medical University, Urumqi 830011, P. R.
China\\
\footnotesize Email: yuehuacl@163.com}
\date{ }
\maketitle

\vspace{-10mm}

\begin{center}
\author{
{\   Qiongxiang
Huang\renewcommand{\thefootnote}{\fnsymbol{footnote}}\footnotemark[1]
}
\\[2mm]
\footnotesize College of Mathematics and System Science,
\footnotesize Xinjiang  University, Urumqi 830046, P. R. China,\\
 \footnotesize Email: huangqx@xju.edu.cn
}
\end{center}

\def\thefootnote{}

\footnotetext {Supported by National Natural Science Foundation of China (Grant Nos. 11671344, 11261059 and 11531011).}

\renewcommand{\thefootnote}{\fnsymbol{footnote}}
\footnotetext[1]{Corresponging author }

\maketitle
\begin{abstract}
An eigenvalue of a graph $G$ is called a main eigenvalue if it has
an eigenvector the sum of whose entries is not equal to zero. It is
well known  that a graph $G$ has exactly two main eigenvalues if and
only if there exists a unique pair of integers $a$ and $b$  such
that $\sum_{u\in N(v)}d(u)=ad(v)+b$  for every vertex $v\in
V(G)$. We collect  such  connected graph $G$ in the set
$\mathscr{G}(a,b)$. In this paper, we mainly focus to the existence of such $a$ and $b$, and give the necessary and sufficient condition for $\mathscr{G}(a,b)\neq\emptyset$.
In addition, we give the bound for the vertex degrees of $G\in\mathscr{G}(a,b)$   and use the bound to characterize the graphs in $\mathscr{G}(a,b)$ for some feasible
pairs $(a,b)$.
\\

\hspace{-6mm}{\it AMS classification:} 05C50\\

\hspace{-6mm}{\it Keywords}: Adjacency matrix; Main eigenvalue;
Equitable Partition;   Bound.
\end{abstract}

\section{\large Introduction}\label{section-1}

Let $G$ be a simple undirected  graph with $n$ vertices, $m$ edges and
 adjacency matrix
$A=A(G)=(a_{ij})_{n\times n}$.  The eigenvalues  of $G$ are
those of $A$. An eigenvalue
 of a graph $G$ is called  \textit{main eigenvalue}
 if it has an eigenvector the sum of
whose entries is not equal to zero. The Perron-Frobenius theorem
implies that the largest eigenvalue  of $G$ is always main. It is well known that a graph is regular if and only if it has
exactly one main eigenvalue. A long standing problem posed by
Cvetkovi\'{c} (see \cite{Cve}) is to
 characterize graphs with exactly $k(k\geq2)$ main
 eigenvalues. There are some literatures on main eigenvalues and one can refer to  a
survey in \cite{Rowlinson}.  Let
$\mathbf{j}$ denote the all-one vector.
In \cite{Hagos}, Hagos derived a simple criterion for a graph to
 have exactly two main eigenvalues.\\
\textbf{Theorem [A]} (\cite{Hagos}){\bf .}
  A graph $G$ has exactly two main
  eigenvalues if and only if there
exists a unique pair of rational numbers $a$ and $b$ such that $A^2\mathbf{j}-aA\mathbf{j}-b\mathbf{j}=\mathbf{0}$.

$A^2\mathbf{j}-aA\mathbf{j}-b\mathbf{j}=\mathbf{0}$ can be
rewritten by
\begin{equation}\label{equ-1}\sum_{u\in N(v)}d(u)=ad(v)+b \ \mbox{ for any $v\in V(G)$.}\end{equation}
In fact,  Hou and  Tian in \cite{Hou1}
showed that   $a$ and $b$ are integers.  In terms of the criterion (\ref{equ-1}),
the connected graphs with $c(G)=n-m+1\le 3$
 that have exactly two main eigenvalues are characterized  in
 \cite{Hou1,Hou2,Shi,Hu,Hou3,Tang1,Fan}.

Let $\mathscr{G}(a,b)$ be the set of  connected graphs having exactly two main eigenvalues and satisfying (\ref{equ-1}). For $G\in\mathscr{G}(a,b)$, let
$\lambda_1,\lambda_2$ are the  main
eigenvalues of $G$. The \textit{main polynomial} of $G$ is defined
by $m_G(x)=(x-\lambda_1)(x-\lambda_2)$.  Teranishi (see
\cite{Rowlinson, Teranishi}) showed that
$f(A)\mathbf{j}=\mathbf{0}$ if and only if $m_G(x)$ divides
$f(x)$, where $f(x)\in \mathbb{R}[x]$. From Theorem A, we know that $A^2\mathbf{j}-a\mathbf{j}-b\mathbf{j}=\mathbf{0}$.
It implies that $m_G(x)|x^2-ax-b$ and so $m_G(x)=x^2-ax-b$. Hence
\begin{equation}\label{equ-2}\lambda_{1}=\frac{a+ \sqrt{a^2+4b}}{2}
\ \ \mbox{and}\ \
\lambda_{2}=\frac{a-\sqrt{a^2+4b}}{2}.\end{equation}

From (\ref{equ-2}) we have
$a^2+4b>0$ and $a=\lambda_1+\lambda_2\geq0$. Set $\mathscr{F}=\{(a,b)\mid a\ge 0 \mbox{ and } a^2+4b>0\}$. Thus, if  $\mathscr{G}(a,b)\not=\emptyset$ then $(a,b)\in\mathscr{F}$. Conversely, we ask wether $(a,b)\in\mathscr{F}$ implies  $ \mathscr{G}(a,b)\not=\emptyset$?

 In Section \ref{section-2}, we give
a necessary and sufficient condition for  $ \mathscr{G}(a,b)\not=\emptyset$.
In Section \ref{section-3}, we obtain an upper and lower bounds for the degrees of vertices in
$G\in\mathscr{G}(a,b)$ and characterize the graphs for which the lower
bound is attained. In Section \ref{section-4}, we characterize the graphs in $\mathscr{G}(a,b)$ for some feasible
pairs $(a,b)$.

\section{\large
The necessary and sufficient condition for  $ \mathscr{G}(a,b)\not=\emptyset$  }\label{section-2}

 Denote by $\delta(G)$, $\Delta(G)$ and
$\overline{d}(G)$ the smallest, largest and mean degrees of $G$,
respectively. First we cite useful lemma due to Rowlinson in \cite{Rowlinson}.

\begin{lem}[\cite{Rowlinson}]\label{E-1}
 Let $G$ be a connected graph with exactly two main eigenvalues
 $\lambda_1$ and $\lambda_2$, where $\lambda_1>\lambda_2$. Then
 $\lambda_2<\delta(G)<\overline{d}(G)<\lambda_1<\Delta(G)$.
\end{lem}

Now we put
$$\mathscr{F}^*=\{(a,b)\not=(0,1)\mid a\ge 0 \mbox{ and } a^2+4b\ge4\},$$ which we call the feasible set of the pair of integers $a$ and
$b$. The following result gives the necessary  condition for   $ \mathscr{G}(a,b)\not=\emptyset$.

\begin{lem}\label{E-3}
If $ \mathscr{G}(a,b)\not=\emptyset$ then $(a,b)\in \mathscr{F}^*$.
\end{lem}

\begin{proof}
Suppose that $ \mathscr{G}(a,b)\not=\emptyset$. From the
above arguments in Section \ref{section-1}, we know that $a\geq 0$
and $a^2+4b> 0$, that is,  $(a,b)\in \mathscr{F}$.

If $a^2+4b=1$ then $a^2=1-4b$ is an odd, and so is $a$. From (\ref{equ-2}), we claim that   $\lambda_{1}=\frac{a+1}{2}$ and
$\lambda_{2}=\frac{a-1}{2}$ are integers.  From Lemma \ref{E-1}, we
have $\frac{a-1}{2}=\lambda_2<\delta(G)<\lambda_1=\frac{a+1}{2}$,
 a contradiction.  If $a^2+4b=2$
then $a^2=2(1-2b)$ is an even, and so is $a$. It follows that
$a^2$ is  a multiple of $4$. This is impossible since $1-2b$ is an
odd. If $a^2+4b=3$ then $a^2=4(1-b)-1$ is an odd, and so is $a$.
Let $a=2k+1$($k\geq 0$), we have $4k^2+4k+1=4(1-b)-1$, which gives that
$k^2+k+b=\frac{1}{2}$,  a contradiction.
Therefore, $a^2+4b\geq 4$.

At last, suppose that $(a,b)=(0,1)$. Let $v\in V(G)$ with $d(v)=\Delta(G)$, and
applying (\ref{equ-1}) to $v$, we have
 $\delta(G)\Delta(G)\leq\sum_{u\in N(v)}d(u)=ad(v)+b=1$, which implies that $G=K_2$,
 and so $G$ has exactly one main eigenvalue, this is a contradiction.

Summarizing the above arguments, we known that $(a,b)\in \mathscr{F}^*$. It completes this proof.
\end{proof}

For a graph $G$, a vertex partition $\pi$: $V(G)=C_1\cup
C_2\cup\cdots\cup C_r$  is said to be \textit{equitable} if, for any
$u\in C_i$, $|C_j\cap N_G(u)|=c_{ij}$  is a constant whenever $1\le
i,j\le r$. Such a partition $\pi$ is called the
\textit{$r$-equitable partition} with cells $C_1,C_2,...,C_r$ and
parameters $(c_{11},c_{12},...,c_{1r};
c_{21},c_{22},...,c_{2r};...;$ $ c_{r1},c_{r2},...,c_{rr})$. An
equitable
 partition  $\pi$ leads to a \textit{quotient
graph}(or \textit{divisor}) which is denoted by $G/\pi$. $G/\pi$ is
a directed graph with  vertex set $\{C_1, C_2,...,C_r\}$, and there
are exactly $c_{ij}$ arcs from  $C_i$ to  $C_j$. Therefore, $G/\pi$
has adjacency matrix $A(G/\pi)=(c_{ij})_{r\times r}$.
 Rowlinson
in \cite{Rowlinson} gave the following result.

\begin{Prop}[\cite{Rowlinson}]\label{E-4}
The main eigenvalues of $G$ are eigenvalues of every divisor of $G$.
\end{Prop}

The following result is a corollary of  Proposition \ref{E-4}.

\begin{Coro} \label{E-5}
If $G$ has a $r$-equitable partition $\pi$, then $G$ has at most $r$ main eigenvalues.
\end{Coro}

 Let $\pi$ be a $r$-equitable partition of $G$ with
cells $C_1,C_2,...,C_r$ and
$P$  be the character matrix whose $i$-th column is   the character
vector of $C_i$ $(1\leq i\leq r)$.  It is easy to see that
 $A(G)P=PA(G/\pi)$. Hence for all $\mathbf{x}\in \mathbb{R}^r$ we have
$$(\lambda I-A(G))P\mathbf{x}=P(\lambda I-A(G/\pi))\mathbf{x}.$$
 Since $P\mathbf{x}=\mathbf{0}$
if and only if $\mathbf{x}=\mathbf{0}$, it follows that $\mathbf{x}$
is an eigenvector of $A(G/\pi)$ with eigenvalue $\lambda$ if and
only if $P\mathbf{x}$ is an eigenvector of $A(G)$  with eigenvalue
$\lambda$. Thus, combining Proposition \ref{E-4} we have the
following result.

\begin{lem} \label{E-5-1}
Let $\pi$ be a $r$-equitable partition of $G$ and $P$ the
character matrix. Then the main eigenvalues of $G$ are the
 eigenvalues of $A(G/\pi)$ which have  an eigenvector
$\mathbf{x}$ such that the sum of the entries of $P\mathbf{x}$ is
not equal to zero.
\end{lem}

Let $G$ be a connected non-regular graph having a  $2$-equitable
partition $\pi:V(G)=C_1\cup C_2$ with parameters  $(c_{11},c_{12};c_{21},c_{22})$. From Corollary \ref{E-5} we know
that $G$ has exactly two main eigenvalues. Clearly, the induced
subgraphs  $G[C_1]$ is $c_{11}$-regular, $G[C_2]$ is
$c_{22}$-regular, and $d_G(u)=c_{11}+c_{12}$ for
 $u\in C_1$,   $d_G(v)=c_{21}+c_{22}$ for $v\in C_2$. Since $G$ is non-regular, we
 have $c_{11}+c_{12}\neq c_{21}+c_{22}$. Applying
(\ref{equ-1}) to $u\in C_1$ and $v\in C_2$, respectively, we have the following linear equations of $a$ and $b$
\begin{equation}\label{equ-3}\left\{\begin{array}{lll}a(c_{11}+c_{12})+b&=c_{11}(c_{11}+c_{12})+c_{12}(c_{21}+c_{22})\\
a(c_{21}+c_{22})+b&=c_{21}(c_{11}+c_{12})+c_{22}(c_{21}+c_{22})\end{array}\right.\end{equation}
which gives $a=c_{11}+c_{22}$ and $b=c_{12}c_{21}-c_{11}c_{22}$,
i.e., $G\in \mathscr{G}(a, b)$. Conversely, if  equations
(\ref{equ-3}) has the unique solution then $c_{11}+c_{12}\neq
c_{21}+c_{22}$. Thus we obtain the following result.

\begin{lem}\label{E-6-1} Let $c_{11}, c_{22}\geq 0$, $c_{12}, c_{21}\geq
1$ be integers. Then equations (\ref{equ-3}) has the unique solution
$a=c_{11}+c_{22}$ and $b=c_{12}c_{21}-c_{11}c_{22}$ if and only if
there exists a graph  $G\in \mathscr{G}(a,b)$   having a $2$-equitable
partition $\pi$ with parameters  $(c_{11},c_{12};c_{21},c_{22})$.
\end{lem}
\begin{proof}
The sufficiency has established from the above arguments. For the
necessity we need to construct a graph $G\in \mathscr{G}(a,b)$ having a $2$-equitable
partition $\pi$ with parameters  $(c_{11},c_{12};c_{21},c_{22})$.

 We first construct  $c_{11}$-regular graph $G_1$
and   $c_{22}$-regular graph $G_2$ such that $|V(G_1)|=|V(G_2)|$ and
 $G_1$  is connected.
 Without loss of generality, let $c_{11}\geq c_{22}\geq 0$. If $c_{11}=0$ then $c_{22}=0$.
 We take $G_1=K_1$ and $G_2=K_1$ in this case.  If $c_{11}=1$ then $c_{22}=0$ or $1$.
 We take $G_1=K_2$, $G_2=2K_1$ or $G_2=K_2$ in this case. If $c_{11}\geq 2$,  let
 $t\geq c_{11}+1$ be an even. We take $G_1=X(\mathbb{Z}_t, S_1)$,  $G_2=tK_1(c_{22}=0)$
 or $G_2=\frac{t}{2}K_2 (c_{22}=1)$ or   $G_2=X(\mathbb{Z}_t, S_2)(c_{22}\geq2)$ in this case,
 where $X(\mathbb{Z}_t, S_i)(i=1,2)$  is circulant graph of order $t$ with inverse-closed
  connection set
  $S_i$ such that $\{1,-1\}\subseteq S_i$ and $|S_i|=c_{ii}$.
 Clearly, $G_1$  is connected since $\{1,-1\}\subseteq S_1$.

Let $V(G_1)=\{u_1,u_2,...,u_t\}$ and $V(G_2)=\{v_1,v_2,...,v_t\}$.
Denote by $c_{21}G_1$ the graph consisting of $c_{21}$ copies of
$G_1$ and $c_{12}G_2$ the graph consisting of $c_{12}$ copies of
$G_2$. Now we construct $G$ (see Figure \ref{fig-1} $(A)$) that is obtained from $c_{21}G_1$ and
$c_{1 2}G_2$ by adding $tc_{12}c_{21}$ edges $u_iv_i(1\leq i\leq t)$
between $c_{21}G_1$ and $c_{12}G_2$. It is easy to verify that $G$
is a connected graph with equitable partition $\pi:V(G) =C_1\cup C_2$,
where $C_1=V(c_{21}G_1)$, $C_2=V(c_{12}G_2)$, and
the parameters of $\pi$ are $(c_{11},c_{12};c_{21},c_{22})$.
Since equations (\ref{equ-3}) has the unique solution,  we have
$c_{11}+c_{12}\neq c_{21}+c_{22}$, which implies that $G$ is
non-regular, and so $G\in \mathscr{G}(a,b)$. Thus $G$ is our required.
\end{proof}

\begin{figure}[h]
\begin{center}
\begin{picture}(376,158)
\put(30,10){\resizebox{11cm}{!}{\includegraphics[0,0][500,100]{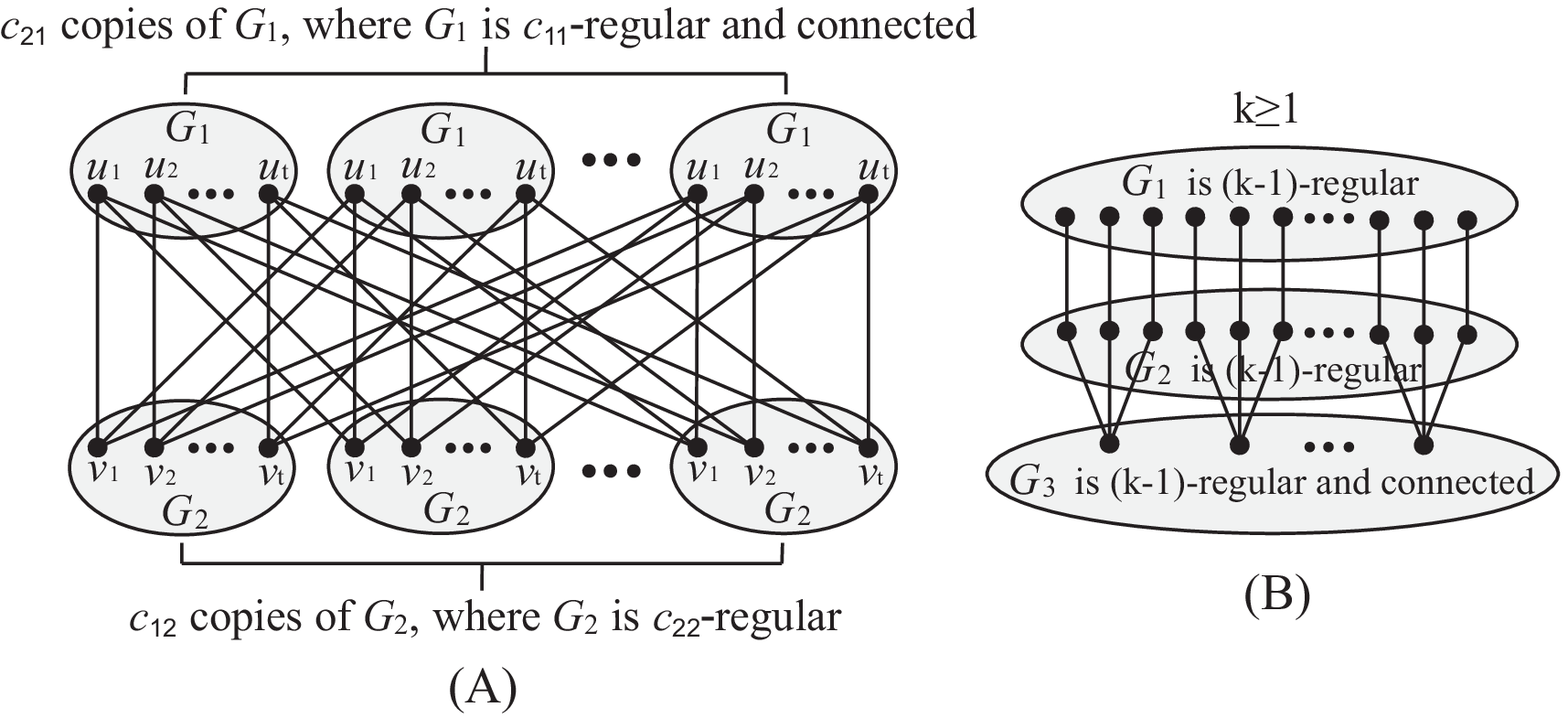}}}
\end{picture}
\end{center}
\vspace{-1.0cm} \caption{ }\label{fig-1} \vspace{-0.2cm}
\end{figure}

\begin{lem}\label{E-6}
For integers $a$ and $b$, if $a\geq 0$ and $a^2+4b>4$ then $ \mathscr{G}(a,b)\not=\emptyset$.
\end{lem}
\begin{proof}
According to Lemma \ref{E-6-1}, it suffices to show that, given $a$
and $b$ satisfying $a\geq 0$ and $a^2+4b>4$, we can find integers
$c_{11}, c_{22}\geq 0$ and $c_{12}, c_{21}\geq 1$  such that
$a=c_{11}+c_{22}$ and $b=c_{12}c_{21}-c_{11}c_{22}$ is the unique
solution of  equations (\ref{equ-3}). We will distinguish two
situations in what follows.

If $a=2k(k\geq0)$, then $b+k^2\geq 2$ since $a^2+4b>4$.   We put
$$\left\{\begin{array}{lll}c_{11}=k,&c_{12}=b+k^2\\
c_{21}=1,&c_{22}=k\end{array}\right.$$
 Then
$c_{11}+c_{12}=b+k^2+k\neq k+1= c_{21}+c_{22}$. Thus equations
(\ref{equ-3}) has unique solution $a=c_{11}+c_{22}$,
$b=c_{12}c_{21}-c_{11}c_{22}$.

If $a=2k+1(k\geq0)$, then $b+k^2+k\geq 1$ since  $a^2+4b>4$. We put
$$\left\{\begin{array}{lll}c_{11}=k+1,&c_{12}=b+k^2+k\\
c_{21}=1,&c_{22}=k\end{array}\right.$$ Then
$c_{11}+c_{12}=b+k^2+2k+1\neq k+1= c_{21}+c_{22}$. Thus equations
(\ref{equ-3}) has unique solution $a=c_{11}+c_{22}$,
$b=c_{12}c_{21}-c_{11}c_{22}$.

 We complete this proof.
\end{proof}

\begin{exm}\label{E-7}
According to the choices of $c_{ij} (1\le i,j\le 2)$ in the proof of Lemma \ref{E-6}, for $0\le a\le5$ and  $b$
with $a^2+4b>4$ we construct some classes of graphs
$H_i(i=1,2,...,6)\in \mathscr{G}(a,b)$ which are depicted in Figure
\ref{fig-2}, where $a=c_{11}+c_{22}$, $b=c_{12}c_{21}-c_{11}c_{22}$
and  the corresponding parameters $c_{ij} (1\le i,j\le 2)$ are shown
in Table 1.
\end{exm}

\begin{center}
\begin{tabular}{|c|c|c|c|c|c|c|l|}
 \multicolumn{8}{c}{ Table 1}\\

\hline &$k$&$c_{11}$&$c_{12}$&$c_{21}$&$c_{22}$& $a$&\hspace{32pt}
$\mathscr{G}(a,b)$\\\hline

$H_1$&$0$&$k$&$b+k^2\geq 2$&$1$&$k$&$0$& $H_1\in \mathscr{G}(0,b), \
\ b\geq2$\\\hline

$H_2$&$1$&$k$&$b+k^2\geq 2$&$1$&$k$&$2$& $H_2\in \mathscr{G}(2,b), \
\
 b\geq 1$\\\hline

$H_3$&$2$&$k$&$b+k^2\geq 2$&$1$&$k$&$4$& $H_3\in \mathscr{G}(4,b), \
\
 b\geq -2$\\\hline

$H_4$&$0$&$k+1$&$b+k^2+k\geq 1$&$1$&$k$&$1$& $H_4\in
\mathscr{G}(1,b), \ \  b\geq 1$\\\hline

$H_5$&$1$&$k+1$&$b+k^2+k\geq 1$&$1$&$k$&$3$& $H_5\in
\mathscr{G}(3,b), \ \ b\geq -1$\\\hline

$H_6$&$2$&$k+1$&$b+k^2+k\geq 1$&$1$&$k$&$5$& $H_6\in
\mathscr{G}(5,b), \ \  b\geq -5$\\\hline

\end{tabular}
\end{center}

\begin{rem}\label{E-6-2}
For some feasible $a=c_{11}+c_{22}$ and
$b=c_{12}c_{21}-c_{11}c_{22}$, the choices of integers $c_{11},
c_{22}\geq 0$ and $c_{12}, c_{21}\geq 1$  are not unique. Besides
the choices  of Lemma \ref{E-6}, we also get  $H_7\in
\mathscr{G}(2,b), H_8\in \mathscr{G}(4,b), H_9\in
\mathscr{G}(3,b)$(see Figure \ref{fig-2}), and the corresponding
parameters $c_{ij}(1\le i,j\le 2)$  are shown in Table 2.
\end{rem}

\begin{center}
\begin{tabular}{|c|c|c|c|c|c|c|l|}
 \multicolumn{8}{c}{ Table 2}\\

\hline &$k$&$c_{11}$&$c_{12}$&$c_{21}$&$c_{22}$& $a$&\hspace{32pt}
$\mathscr{G}(a,b)$ \\\hline

$H_7$&$1$&$k+1$&$b+k^2-1\geq 1$&$1$&$k-1$&$2$& $H_7\in
\mathscr{G}(2,b), \ \  b\geq 1$\\\hline

$H_8$&$2$&$k+1$&$b+k^2-1\geq 1$&$1$&$k-1$&$4$& $H_8\in
\mathscr{G}(4,b), \ \  b\geq -2$\\\hline

$H_9$&$1$&$k+2$&$b+k^2+k-2\geq 1$&$1$&$k-1$&$3$& $H_9\in
\mathscr{G}(3,b), \ \  b\geq 1$\\\hline

\end{tabular}
\end{center}

\begin{figure}[h]
\begin{center}
\begin{picture}(396,196)
\put(30,10){\resizebox{11cm}{!}{\includegraphics[0,0][500,100]{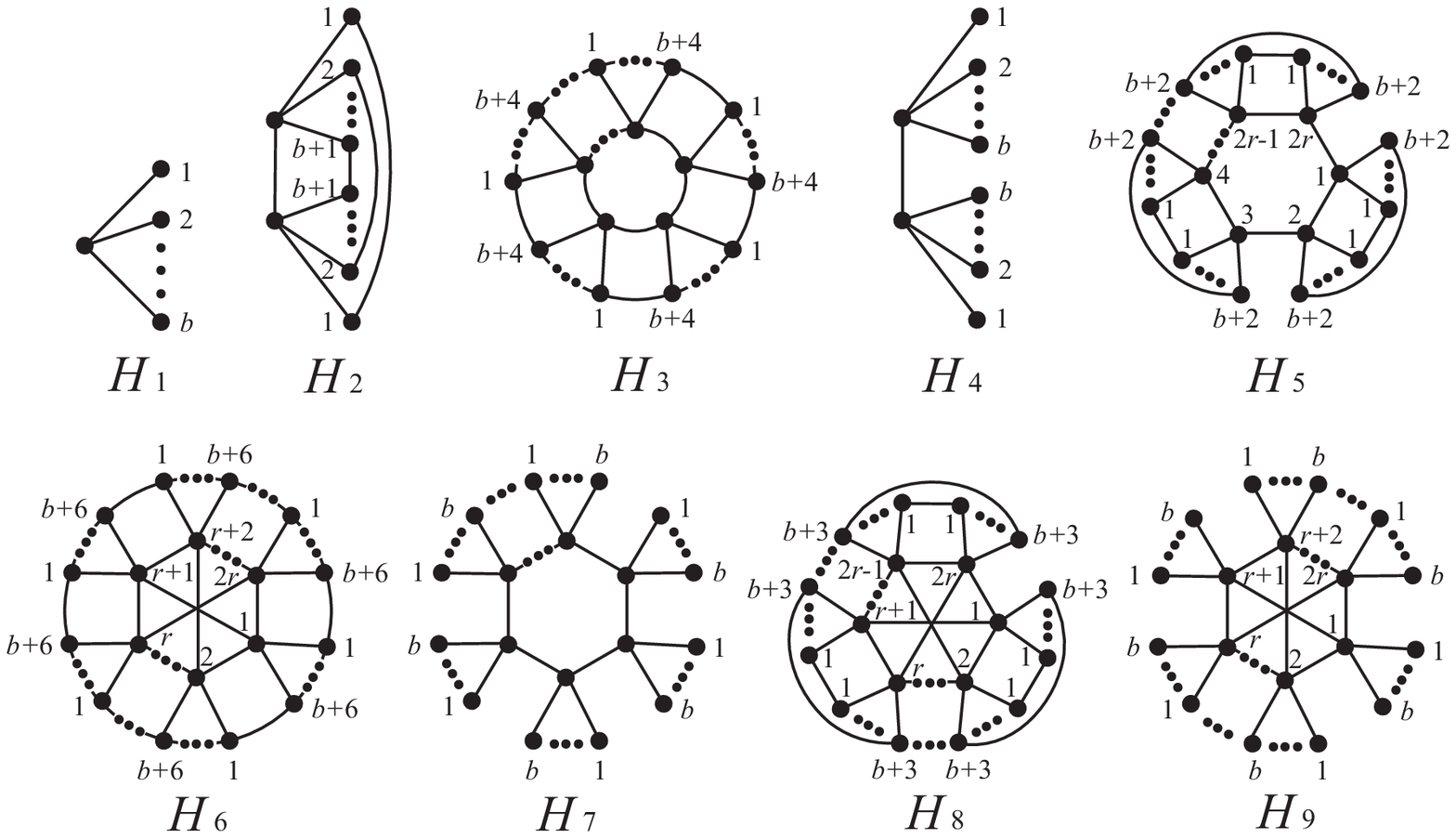}}}
\end{picture}
\end{center}
\vspace{-1.0cm} \caption{ }\label{fig-2} \vspace{-0.2cm}
\end{figure}

\begin{Prop}\label{E-8}
Let $G\in \mathscr{G}(a,b)$, where $a\geq 0$, $a^2+4b=4$ and
$(a,b)\neq (0,1)$. Then any equitable partition of $G$ has at least
three cells.
\end{Prop}
\begin{proof}
By assumption,    $a^2=4(1-b)$ is an even, and so is $a$. Let
$a=2k(k\geq 0)$, we have $k\geq1$ since $(a,b)=(0,1)$ if $k=0$.
 From Lemma
\ref{E-1} and (\ref{equ-2}) we know that
$k-1=\lambda_2<\delta(G)<\lambda_1=k+1$, and so $\delta(G)=k$.

On the contrary we may
 suppose that $G$ has a $2$-equitable partition $\pi:V(G)=C_1\cup C_2$ with parameters $(c_{11},c_{12}; c_{21},c_{22})$.  From Lemma \ref{E-6-1} we know that $a=c_{11}+c_{22}$ and $b=c_{12}c_{21}-c_{11}c_{22}$, and   $c_{12}, c_{21}\geq 1$ since $G$ is
connected.
 Without loss of generality, we assume that
 $d(u)=\delta(G)$ for some  $u\in C_1$. Then $k=c_{11}+c_{12}$, and so $1\leq c_{12}\leq k$. We have
$$1\leq c_{21}=\frac{b+c_{11}c_{22}}{c_{12}}=\frac{(1-\frac{1}{4}a^2)+c_{11}(a-c_{11})}{c_{12}}=\frac{(1-k^2)+(k-c_{12})(k+c_{12})}{c_{12}}=\frac{1-c_{12}^2}{c_{12}}.$$
It is impossible.

We complete this proof.
\end{proof}

\begin{lem}\label{E-9}
For integers $a$ and $b$, if $a\geq 0$,  $a^2+4b=4$ and $(a,b)\neq(0,1)$, then $ \mathscr{G}(a,b)\not=\emptyset$.
\end{lem}
\begin{proof}
Note that $a\geq 0$, $a^2+4b= 4$ and
$(a,b)\neq(0,1)$ if and only if  $a=2k$ and $b=1-k^2$ for $k\geq 1$. It
suffices to show  $
 \mathscr{G}(2k,1-k^2)\not=\emptyset$ for $k\geq 1$.
 By Proposition \ref{E-8},    $G\in \mathscr{G}(2k,1-k^2)$ has no $2$-equitable partition.
 In what follows we will construct $G\in \mathscr{G}(2k,1-k^2)$
 with a $3$-equitable partition $\pi:V(G)=C_1\cup C_2\cup C_3$.

First  we can construct
three  $(k-1)$-regular graphs $G_1$, $G_2$ and $G_3$ with vertex
sets $V(G_1)=\{u_{ij}|i=1,2,3; j=1,2,...,t\}$,
$V(G_2)=\{v_{ij}|i=1,2,3; j=1,2,...,t\}$ and $V(G_3)=\{w_j|
j=1,2,...,t\}$ such that $G_3$ is connected ($G_3$ will be $K_1$ if $k=1$).  Now we construct $G$ (see Figure \ref{fig-1} $(B)$)
 from $G_1, G_2$ and $G_3$ by adding edges
$\{u_{ij}v_{ij}\mid 1\leq i\leq 3, 1\leq j\leq t\}$ and edges
$\{v_{ij}w_j\mid 1\leq i\leq 3, 1\leq j\leq t\}$.  It is easy to
verify that $G$ is a connected graph with equitable partition $\pi:V(G)
=C_1\cup C_2\cup C_3$, where $C_1=V(G_1)$, $C_2=V(G_2)$,
$C_3=V(G_3)$, and the corresponding parameters $c_{ij}$ are
$$\label{equ-4}\left\{\begin{array}{lll}c_{11}=k-1,& c_{12}=1,&c_{13}=0,\\
c_{21}=1,&c_{22}=k-1,&c_{23}=1,\\
c_{31}=0,&c_{32}=3,&c_{33}=k-1.\end{array}\right.$$ By direct
calculation we  know that the eigenvalues of $A(G/\pi)$ are
$\lambda_1=k+1$, $\lambda_2=k-1$ and $\lambda_3=k-3$ with respect to
eigenvectors $\mathbf{x}_1=(1,2,3)^T$, $\mathbf{x}_2=(-1,0,1)^T$ and
$\mathbf{x}_3=(1,-2,3)^T$, respectively. Let $P$ be the
character matrix of $\pi$. It is easy to verify that  the sums
of the entries of $P\mathbf{x}_1$, $P\mathbf{x}_2$ and
$P\mathbf{x}_3$ are $3t\times1+3t\times2+t\times3\neq0$,
$3t\times(-1)+3t\times0+t\times1\neq0$ and
$3t\times1+3t\times(-2)+t\times3=0$. From Lemma
\ref{E-5-1} we know that $G$ has exactly two main eigenvalues
$\lambda_1=k+1$ and $\lambda_2=k-1$, and so  $G\in
\mathscr{G}(a,b)$, where $a=\lambda_1+\lambda_2=2k$ and
$b=-\lambda_1\lambda_2=1-k^2$.

We complete this proof.
\end{proof}

\begin{exm}\label{E-10}

According to the choices of $c_{ij} (1 \le i,j\le 3)$ in the proof of Lemma \ref{E-9},  for $a=2k$ and
$b=1-k^2(k=1,2,3)$ we construct three  graphs $H_i(i=10,11,12)\in
\mathscr{G}(a,b)$ which are depicted  in Figure \ref{fig-3}, and the
corresponding parameters $c_{ij} (1\le i,j\le 3)$ are shown
in Table 3.
\end{exm}

\begin{center}
\begin{tabular}{|c|c|c|c|c|c|c|c|c|c|c|c|c|l|}
 \multicolumn{14}{c}{ Table 3}\\

\hline
&$k$&$c_{11}$&$c_{12}$&$c_{13}$&$c_{21}$&$c_{22}$&$c_{23}$&$c_{31}$&$c_{32}$&$c_{33}$&
$a$&$b$&$\hspace{20pt}\mathscr{G}(a,b)$
\\\hline

$H_{10}$&$1$&$k-1$&$1$&$0$&$1$&$k-1$&$1$&$0$&$3$&$k-1$&$2$&$0$&$H_{10}\in
\mathscr{G}(2,0)$\\\hline

$H_{11}$&$2$&$k-1$&$1$&$0$&$1$&$k-1$&$1$&$0$&$3$&$k-1$&$4$&$-3$&$H_{11}\in
\mathscr{G}(4,-3)$\\\hline

$H_{12}$&$3$&$k-1$&$1$&$0$&$1$&$k-1$&$1$&$0$&$3$&$k-1$&$6$&$-8$&$H_{12}\in
\mathscr{G}(6,-8)$\\\hline

\end{tabular}
\end{center}

\begin{rem}\label{E-9-1}
The choices of  $c_{ij} (1\le i,j\le 3)$ in the proof of Lemma \ref{E-9}
are not unique for $k\geq 2$. Besides the choices  of Lemma \ref{E-9},
we also get  $H_{13}\in \mathscr{G}(4,-3), H_{14}\in
\mathscr{G}(6,-8)$(see Figure \ref{fig-3}), and the corresponding
parameters $c_{ij} (1\le i,j\le 3)$  are shown in Table 4.
\end{rem}

\begin{center}
\begin{tabular}{|c|c|c|c|c|c|c|c|c|c|c|c|c|l|}
 \multicolumn{14}{c}{ Table 4}\\

\hline
&$k$&$c_{11}$&$c_{12}$&$c_{13}$&$c_{21}$&$c_{22}$&$c_{23}$&$c_{31}$&$c_{32}$&$c_{33}$&
$a$&$b$&$\hspace{20pt}\mathscr{G}(a,b)$
\\\hline

$H_{13}$&$2$&$k-1$&$1$&$0$&$2$&$k-2$&$1$&$0$&$4$&$k-1$&$4$&$-3$&$H_{13}\in
\mathscr{G}(4,-3)$\\\hline

$H_{14}$&$3$&$k-1$&$1$&$0$&$2$&$k-2$&$1$&$0$&$4$&$k-1$&$6$&$-8$&$H_{14}\in
\mathscr{G}(6,-8)$\\\hline

\end{tabular}
\end{center}

\begin{figure}[h]
\begin{center}
\begin{picture}(278,150)
\put(30,10){\resizebox{11cm}{!}{\includegraphics[0,0][500,100]{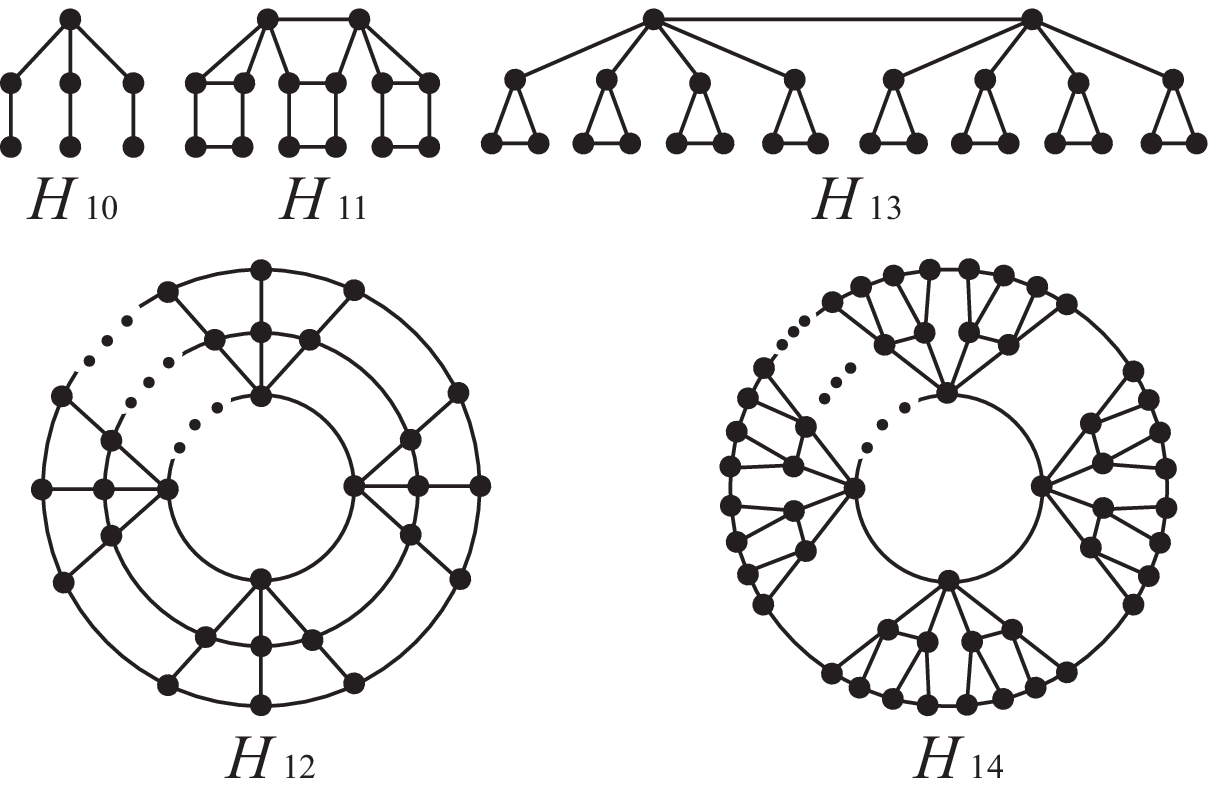}}}
\end{picture}
\end{center}
\vspace{-0.8cm} \caption{ }\label{fig-3} \vspace{-0.2cm}
\end{figure}

From Lemmas \ref{E-3}, \ref{E-6} and \ref{E-9}, we have the
following theorem.

\begin{thm}\label{E-10}
$ \mathscr{G}(a,b)\not=\emptyset$ if and only if $(a,b)\in \mathscr{F}^*$.
\end{thm}
It is worth to mention that our proof of Theorem \ref{E-10} is made of constructions. It implies that, for any pair of $(a,b)\in \mathscr{F}^*$, we can construct a graph $G\in\mathscr{G}(a,b)$ such that it has $2$, or $3$-equitable
partition. In the next section, we will give
an upper and lower bounds for the degrees of vertices in
$G\in\mathscr{G}(a,b)$ and  characterize graphs for which the lower
bound is attained.

\section{\large
The bound for the vertex degrees of
$G\in \mathscr{G}(a,b)$  }\label{section-3}

In \cite{Tang2}, Tang and Hou gave the following result.
\begin{lem}[\cite{Tang2}]\label{E-25-1}
Let $G\in \mathscr{G}(a,b)$. If $a\geq 1$ then
\begin{equation}\label{equ-4}\triangle(G)\leq \frac{a^2-a+b+1+\sqrt{(a^2-a+b+1)^2+4(a-1)b}}{2}\end{equation}
\end{lem}
For $G\in \mathscr{G}(a,b)$, let us define two numbers $\delta^*(G)$ and $\Delta^*(G)$ bellow:
\begin{equation*}\begin{array}{ll}\delta^*(G)=\frac{a^2-a\delta(G)+b+\delta(G)-\sqrt{(a^2-a\delta(G)+b+\delta(G))^2+4(a-\delta(G))b}}{2},\\
\Delta^*(G)=\frac{a^2-a\delta(G)+b+\delta(G)+\sqrt{(a^2-a\delta(G)+b+\delta(G))^2+4(a-\delta(G))b}}{2}.\end{array}\end{equation*}
The following result gives the lower and upper
bounds for the degrees of vertices
 in $G\in\mathscr{G}(a,b)$,  which  improves the result  of Lemma
\ref{E-25-1}.

\begin{lem}\label{E-25}
Let $G\in \mathscr{G}(a,b)$.  For any $u\in V(G)$ we have
 $\delta^*(G)\leq d(u)\leq \Delta^*(G)$. Furthermore,
 $d(u)$ achieves its upper bound  or  lower bound if and only if
 $d(x)=\delta(G)$ for any $x\in (\cup_{v\in N(u)}
N(v))\backslash \{u\}$ (if any).
\end{lem}
\begin{proof}
Applying (\ref{equ-1}) to $v\in N(u)$ we have
\begin{equation}\label{equ-6}ad(v)+b=\sum_{x\in N(v)}d(x)\geq
d(u)+(d(v)-1)\delta(G).\end{equation} It follows that
$(a-\delta(G))d(v)\geq d(u)-\delta(G)-b$. Thus
$$\begin{array}{ll}
(a-\delta(G))(ad(u)+b)&=(a-\delta(G))\sum_{v\in N(u)}d(v)\\
&=\sum_{v\in N(u)}(a-\delta(G))d(v)\\
&\geq d(u)(d(u)-\delta(G)-b),\\
\end{array}$$
which gives
\begin{equation}\label{equ-7}d(u)^2-(a^2-a\delta(G)+\delta(G)+b)d(u)-(a-\delta(G))b\leq0,
\end{equation}
$\delta^*(G)\leq d(u)\leq \Delta^*(G)$  follows immediately by solving the inequality
(\ref{equ-7}).

If $d(u)= \Delta^*(G)$ (or $d(u)= \delta^*(G)$),  then
(\ref{equ-7}), and so (\ref{equ-6}) will be  equality, which implies
that $d(x)=\delta(G)$ for any $x\in (\cup_{v\in N(u)}
N(v))\backslash \{u\}$. Conversely, if $d(x)=\delta(G)$ for any
$x\in (\cup_{v\in N(u)} N(v))\backslash \{u\}$ then (\ref{equ-6}),
and  subsequently (\ref{equ-7})  must be  equality, which implies
that $d(u)= \Delta^*(G)$ or $d(u)=\delta^*(G)$.

We complete this proof.
\end{proof}

From Lemma \ref{E-25}, it immediately follows the boundsfor the smallest and largest  degrees of  $G\in \mathscr{G}(a,b)$,
$$\delta^*(G)\leq \delta(G)\ \ \mbox{and}\ \ \Delta(G)\leq \Delta^*(G).$$
The first (the second) equality holds if and only if there
exists a vertex $u\in V(G)$ such that $d(u)=\delta^*(G)$
($d(u)=\Delta^*(G)$). Additionally, $\Delta(G)\leq
\Delta^*(G)$ leads to (\ref{equ-4}) if $\delta(G)=1$.

 Let
$\mathscr{G}_{\Delta^*}(a,b)=\{G\in\mathscr{G}(a,b)|
\Delta(G)=\Delta^*(G)\}$ and $\mathscr{G}_{\delta^*}(a,b)=\{G\in\mathscr{G}(a,b)|
\delta(G)=\delta^*(G)\}$.  Next, we will characterize graphs in $\mathscr{G}_{\delta^*}(a,b)$.

In \cite{Rowlinson}, Rowlinson gave the following result.

\begin{lem}[\cite{Rowlinson}]\label{E-22}
Let $G$ be a non-trivial connected graph with index $\lambda$. Then $G$ is
a semi-regular bipartite graph if and only if the main eigenvalues of $G$ are $\lambda$ and
$-\lambda$.
\end{lem}

By Lemma \ref{E-22}, the following result can be obtained directly.

\begin{thm}\label{E-23}
$G\in \mathscr{G}(0,b)$ if and only if $G$ is a connected semi-regular bipartite
 graph with $\delta(G)\Delta(G)=b$, where $b\geq 2$.
\end{thm}

\begin{thm}\label{E-25-2}
$\mathscr{G}_{\delta^*}(a,b)=\mathscr{G}(0,b)$.
\end{thm}

\begin{proof}
 If $G\in \mathscr{G}(0,b)$ then $\delta(G)\Delta(G)=b$ by  Theorem \ref{E-23}, which gives $\delta(G)<b$, and so
$\delta^*(G)=\frac{b+\delta(G)-|b-\delta(G)|}{2}=\delta(G)$, that is $G\in
\mathscr{G}_{\delta^*}(a,b)$. Conversely, if $G\in
\mathscr{G}_{\delta^*}(a,b)$ then $\delta(G)=\delta^*(G)$, that is
$$\delta(G)=\frac{a^2-a\delta(G)+b+\delta(G)-\sqrt{(a^2-a\delta(G)+b+\delta(G))^2+4(a-\delta(G))b}}{2},$$
and so
$$(a^2-a\delta(G)+b-\delta(G))^2-(a^2-a\delta(G)+b+\delta(G))^2-4(a-\delta(G))b=0.$$
It follows that $4a(a\delta(G)+b-\delta(G)^2)=0$, which implies that
$a=0$, and so $G\in \mathscr{G}(0,b)$. Otherwise,
 $a\delta(G)+b-\delta(G)^2=0$.   Let $u\in V(G)$ and
$d(u)=\delta(G)$.  By applying (\ref{equ-1}) to $u$,  we have
$\sum_{v\in N(u)}d(v)=a\delta(G)+b=\delta(G)^2$, which leads to
$d(v)=\delta(G)$ for any $v\in N(u)$.  By regarding  $v\in N(u)$ as
$u$, and repeating  this procedure, we know that $G$ is a
$\delta(G)$-regular graph, which contradicts $G\in
\mathscr{G}(a,b)$.

We complete this proof.
\end{proof}

At the last of this section we propose a problem to characterize the graphs in
$\mathscr{G}_{\Delta^*}(a,b)$.

\section{\large
Characterize the graphs in $\mathscr{G}(a,b)$ for some
$(a,b)\in \mathscr{F}^*$} \label{section-4}

Let $\mathscr{G}$ be the set of the connected graphs having
exactly two main eigenvalues. The graphs in $\mathscr{G}$ would be
partitioned as trees, unicyclic, bicyclic and tricyclic graphs,...,
and so on. The graphs in  $\mathscr{G}$ are characterized on this
way by many authors and  come up against tricyclic graphs (to refer
\cite{Hou1,Hou2,Shi,Hu,Hou3,Tang1,Fan} for references).

Theorem \ref{E-10} leads to a new  partition: $\mathscr{G}=\cup_{(a,b)\in \mathscr{F}^*}\mathscr{G}(a,b)$.
 It provides us a
classification of   $\mathscr{G}$ such that we can advance forward
this work. $\mathscr{G}(0,b)(b\geq 2)$ is determined in Theorem
\ref{E-23}.
In this section we will characterize the graphs in $\mathscr{G}(2,0)$ and  $\mathscr{G}(1,b)$ for $b=1,2,3,4,5$.

Let $a(a\geq 2)$ be an integer. Denote by $T_{a}$ the rooted tree
with root $u$, where $d(u)=a^2-a+1$, $d(v)=a$ for
any $v\in N(u)$ and $d(x)=1$ for any $x\in V(T_a)\setminus(
\{u\}\cup N(u))$. For example, $H_{10}$ depicted in  Figure
\ref{fig-3} is $T_2$. Using (\ref{equ-1}), it is easy to verify that $T_a\in
\mathscr{G}(a,0) $.

\begin{thm}\label{E-38}
 $\mathscr{G}(2,0)=\{T_2\}$.
\end{thm}
\begin{proof}
First we know that $T_2\in
\mathscr{G}(2,0) $. Let $G\in \mathscr{G}(2,0)$, from Lemma
\ref{E-1} and (\ref{equ-2}),  we have
$0=\lambda_2<\delta(G)<\lambda_1=2<\Delta(G)$, which gives $\delta(G)=1$ and $\Delta(G)\geq 3$. By simple calculation, we have $\Delta^*(G)=3$, and so $\Delta(G)=3$.
 Let $u\in V(G)$, $d(u)=3$ and $N(u)=\{v_1,v_2,v_3\}$.
 Applying (\ref{equ-1}) to $u$, we
have $d(v_1)+d(v_2)+d(v_3)=6$, which implies that $d(v_1)=d(v_2)=d(v_3)=2$. Otherwise, there exists a vertex in $N(u)$, say $v_1$, with $d(v_1)=1$. Applying (\ref{equ-1}) to $v_1$, we
have $d(u)=2$, a contradiction. Let $w_i=N(v_i)\setminus \{u\}(i=1,2,3)$,  we
have $d(w_i)=\delta(G)=1$ by Lemma \ref{E-25}. Thus $G=T_2$.
\end{proof}

Let $G\in \mathscr{G}(1,b)$, from Theorem \ref{E-10} we know that
$a^2+4b\geq 4$, which gives
  $b\geq 1$.

\begin{lem}\label{E-29}
Let  $G\in \mathscr{G}(1,b)$. If $b\in \{1,2\}$ then $\delta(G)=1$;
if $b\in \{3,4,5,6\}$ then $\delta(G)\leq2$.
\end{lem}
\begin{proof}
 From Lemma
\ref{E-1} and (\ref{equ-2}),  we have $\delta(G)<\lambda_1=\frac{1+
\sqrt{1^2+4b}}{2}$. By replacing $b$ in $\frac{1+ \sqrt{1^2+4b}}{2}$
with $1,2,3,4,5$ and $6$, respectively, our result will be verified.
\end{proof}

 A double star $S_{n_1,n_2}$ is a graph of order $n_1+n_2+2$
obtained from an edge $uv$ by adding $n_1+n_2$ pendant edges
$uu_1,uu_2,...,uu_{n_1},vv_1,vv_2,...,vv_{n_2}$,  where $n_1\geq1$
and $n_2\geq1$.

\begin{lem}\label{E-30}
If $G\in \mathscr{G}(1,b)$ and $\delta(G)=1$, then $G=S_{b,b}$.
\end{lem}
\begin{proof}
By simple calculation, we have $\Delta^*(G)=1+b$. Let $v_1\in V(G)$  be a pendant vertex and  $u$ be the unique
neighbor of $v_1$. Applying (\ref{equ-1}) to $v_1$, we have
$d(u)=1+b=\Delta^*(G)$. Let $N(u)=\{v_1,v_2,...,v_{1+b}\}$. Applying
(\ref{equ-1}) to $u$,  we have $\sum_{i=1}^{1+b}d(v_i)=1+2b$, which
implies that  there exists a vertex in $N(u)$, say $v_{1+b}$, with $d(v_{1+b})\ge2$. Let $w_1\in N(v_{1+b})\setminus \{u\}$,  we
have $d(w_1)=1$ by Lemma \ref{E-25}, and again get $d(v_{1+b})=1+b$ by applying (\ref{equ-1}) to $w_1$. Since $\sum_{i=1}^{1+b}d(v_i)=1+2b$, we have
$d(v_1)=d(v_2)=\cdots =d(v_{b})=1$. Let
$N(v_{1+b})=\{u,w_1,w_2,...,w_{b}\}$. Similarly, we obtain
$d(w_1)=d(w_2)=\cdots =d(w_b)=1$ by regarding $v_{1+b}$ as $u$ in above arguments. Hence $G=S_{b,b}$.
\end{proof}

Denote by $\mathscr{G}_1(1,b) (b\geq 1)$  the set  of all the
connected graphs in which every graph has a $2$-equitable partition
$\pi: V(G)=C_1\cup C_2$ with  parameters $(1,i;j,0)$, where $i,j\geq 1$, $ ij=b$ and $j\neq 1+i$.  Using
(\ref{equ-1}), it is easy to  verify that
$\mathscr{G}_1(1,b) \subseteq \mathscr{G}(1,b)$.

Obviously,
$S_{b,b}\in \mathscr{G}_1(1,b) $ and the corresponding parameters are $(1,b;1,0)$.

\begin{thm}\label{E-34}
$\mathscr{G}(1,1)=\{S_{1,1}\}$, $\mathscr{G}(1,2)=\{S_{2,2}\}$.
\end{thm}

\begin{proof}
First we know that $S_{b,b}\in\mathscr{G}_1(1,b)\subseteq
\mathscr{G}(1,b)$. Let $G\in \mathscr{G}(1,b)$ and $b\in \{1,2\}$.
 Then $\delta(G)=1$ by Lemma \ref{E-29},
and so $G=S_{b,b}$ by Lemma \ref{E-30}.
\end{proof}

For $b\geq 3$, let $\mathcal {G}_1(1,b)=\{G\in  \mathscr{G}_1(1,b)| (1,i;j,0)=(1,1;b,0)\}$. Let $\mathcal {G}_2(1,4)=\{G\in  \mathscr{G}_1(1,4)| (1,i;j,0)=(1,2;2,0)\}$. We have the following result.

\begin{lem}\label{E-33}
If $G\in \mathscr{G}(1,b)$ and  $\delta(G)=2$, then
$1+\frac{b}{2}\leq \Delta(G)\leq b$. In particular, if $\Delta(G)=b$
then $G\in \mathcal {G}_1(1,b)$.
\end{lem}

\begin{proof}
First we claim that $b\geq 3$, for otherwise $\delta(G)=1$ by Lemma  \ref{E-29}. Let
$v\in V(G)$ and $d(v)=2$. Applying (\ref{equ-1}) to $v$, we have
$2+b=\sum_{u\in N(v)}d(u)\leq 2\Delta(G)$, which gives
$1+\frac{b}{2}\leq\Delta(G)$. On the other hand, from Lemma
\ref{E-25} we have $\Delta(G)\leq \Delta^*(G)=
\frac{b+1+\sqrt{(b-1)^2}}{2}=b$.

If $\Delta(G)=b$, let  $u\in V(G)$, $d(u)=b$ and
$N(u)=\{v_1,v_2,....,v_b\}$. Applying (\ref{equ-1}) to $u$,  we have
$\sum_{i=1}^{b}d(v_i)=2b$, which implies that $d(v_1)=d(v_2)=\cdots
=d(v_b)=2$ since $\delta(G)=2$. Let $N(v_i)=\{u,w_i\}(i=1,2,...,b)$.
Since $d(u)=\Delta^*(G)$, from Lemma \ref{E-25} we have
$d(w_1)=d(w_2)=\cdots =d(w_b)=\delta(G)=2$. Let
$N(w_i)=\{v_i,x_i\}(i=1,2,...,b)$.
 Applying (\ref{equ-1}) to $w_i$, we have $2+b=ad(w_i)+b=d(v_i)+d(x_i)=2+d(x_i)$, which gives $d(x_1)=d(x_2)=\cdots =d(x_b)=b$.
  By regarding $x_i$ as $u$,  and repeating   above procedure, we know that $d(y)=2$ or $b$
   for any $y\in V(G)$, and that   if $d(y)=b$ then $d(z)=2$ for any $z\in N(y)$;
   if $d(y)=2$  then the degrees of the two neighbors of $y$ are $2$ and $b$, respectively. Let
 $$
\begin{array}{ll}C_1=\{u\in V(G)|d(u)=2 \}, \\
C_2=\{u\in V(G)|d(u)=b \}.\\
 \end{array}
$$
Then $\pi:V(G)=C_1\cup C_2$ is a partition of $V(G)$. From  the  above
arguments we know that $|N(u)\cap C_1|=1$ and $|N(u)\cap C_2|=1$ for
any $u\in C_1$, $|N(u)\cap C_1|=b$ and $|N(u)\cap C_2|=0$ for any
$u\in C_2$. Therefore, $\pi:V(G)=C_1\cup C_2$ is a $2$-equitable partition
of $V(G)$  with  parameters $(1,1;b,0)$. By the definition,  $G\in \mathcal {G}_1(1,b)$.
\end{proof}

\begin{figure}
\begin{center}
\begin{picture}(380,88)
\put(30,10){\resizebox{11cm}{!}{\includegraphics[0,0][500,100]{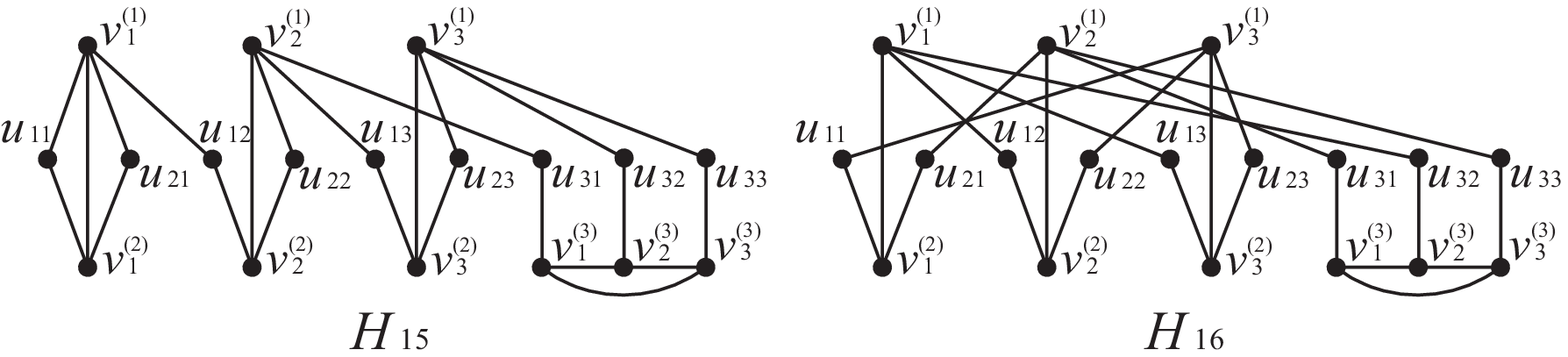}}}
\end{picture}
\end{center}
\vspace{-0.8cm} \caption{ }\label{fig-4} \vspace{-0.2cm}
\end{figure}

Given $t\ge3 $, now we define a   graph $G=(V(G),E(G))$ with a
partition \begin{equation}\label{equ-16-0}V(G)=V_1\cup V_2\cup V_3
\cup V_4, \end{equation}
 where
$V_1=\{v_1^{(1)},v_2^{(1)},...,v_t^{(1)}\}$,
$V_2=\{v_1^{(2)},v_2^{(2)},...,v_t^{(2)}\}$,
$V_3=\{v_1^{(3)},v_2^{(3)},...,v_t^{(3)}\}$ and    $V_4$ has exactly
$3t$ vertices  which are decomposed into three  parts of the same
cardinality: $V_{i4}'=\{u_{i1}',u_{i2}',...,u_{it}'\}$ for
$i=1,2,3$. $G[V_i]$ is a vertex independent set for $i=1,2,4$,
$G[V_3]$ is a
 $2$-regular subgraph. The edges of $G$ between $V_1$ and $V_4$ are defined
to be $$\{v_i^{(1)}u_{ji}'|i=1,2,...,t; j=1,2,3\}. $$ Relabeling the
vertices in $V_4$ arbitrarily such that they are decomposed into
another three parts of the same cardinality:
$V_{i4}=\{u_{i1},u_{i2},...,u_{it}\}$ for $i=1,2,3$. The other edges
of $G$ are defined to be
$$ \left\{\begin{array}{ll}
\mbox{ the edges between $V_1$ and $V_2$: $\{v_i^{(1)}v_i^{(2)}\mid i=1,...,t\}$},\\
\mbox{ the edges between $V_2$ and $V_4$: $\{ v_i^{(2)}u_{1i},  v_i^{(2)}u_{2i}\mid i=1,...,t\}$},\\
\mbox{ the edges between $V_3$ and $V_4$: $\{ v_i^{(3)}u_{3i}\mid i=1,...,t\}$}.\\
\end{array}\right.
$$
We collect such connected graph  $G$ in the set $\mathscr{G}_2(1,5)$. For
example, the graphs $H_{15}$ and $H_{16}$  depicted in Figure
\ref{fig-4} are in $\mathscr{G}_2(1,5)$.

By definition, $G\in \mathscr{G}_2(1,5)$ has vertex partition
(\ref{equ-16-0}),  and  $d(v_i)$ and $\sum_{u\sim v_i}d(u)$ are
constant for any $v_i\in V_i$ (for example, any $v_1\in V_1$ has
four neighbors whose degrees are $2$, $2$, $2$ and $3$,
respectively, and so $\sum_{u\sim v_1}d(u)=3\cdot 2+3=9$). Now we
take $v_i\in V_i$ for $i=1,2,3,4$, and put
$$\left\{\begin{array}{ll}
a\cdot 4+b=ad(v_1)+b=\sum_{u\in N(v_1)}d(u)=3\cdot 2+3=9,\\
a\cdot 3+b=ad(v_2)+b=\sum_{u\in N(v_2)}d(u)=2\cdot 2+4=8,\\
a\cdot 3+b=ad(v_3)+b=\sum_{u\in N(v_3)}d(u)=2\cdot 3+2=8,\\
a\cdot 2+b=ad(v_4)+b=\sum_{u\in N(v_4)}d(u)=3+4=7.\\
\end{array}\right.
$$
The above equation gives an unique solution $a=1, b=5$. It
immediately follows the following result by (\ref{equ-1}).

\begin{lem}\label{36}
 $\mathscr{G}_2(1,5)\subseteq \mathscr{G}(1,5)$.
\end{lem}

\begin{figure}[h]
\begin{center}
\begin{picture}(330,98)
\put(38,10){\resizebox{11cm}{!}{\includegraphics[0,0][500,100]{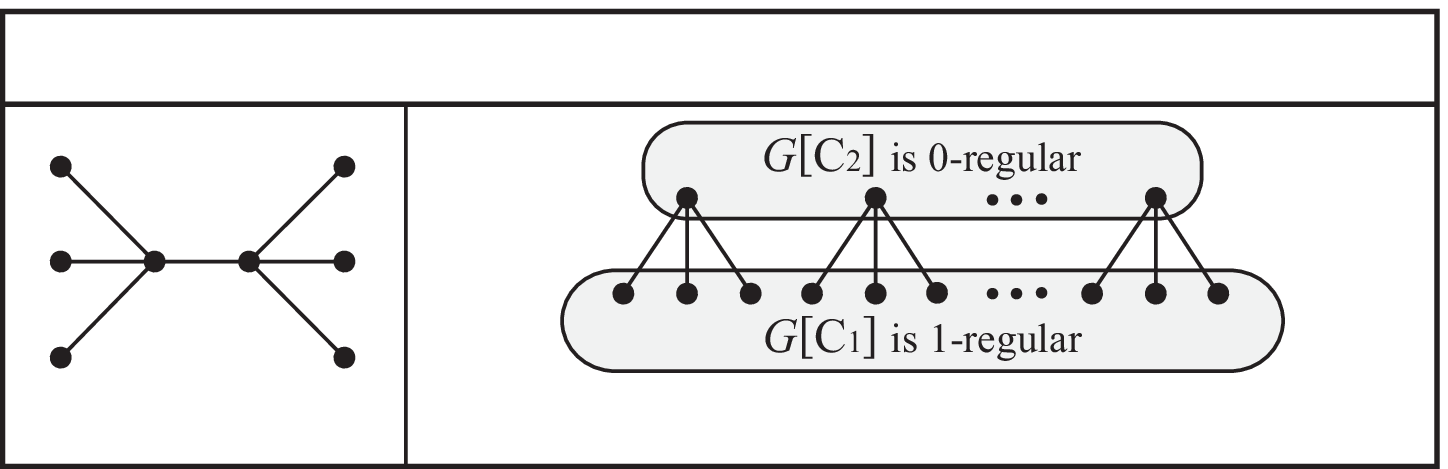}}}

\put(69,16){\small $S_{3,3}$}

\put(125,16){\small We collect such graph $G$ in $\mathcal {G}_1(1,3)$}

\put(90,81.2){\small$\mathscr{G}(1,3)=\mathscr{G}_1(1,3)=\{S_{3,3}\}\cup \mathcal {G}_1(1,3)$}

\end{picture}
\end{center}
\vspace{-1.0cm} \caption{All the graphs in $\mathscr{G}(1,3)$ }\label{fig-5} \vspace{-0.2cm}
\end{figure}

\begin{figure}[h]
\begin{center}
\begin{picture}(550,96)
\put(44,10){\resizebox{11cm}{!}{\includegraphics[0,0][500,100]{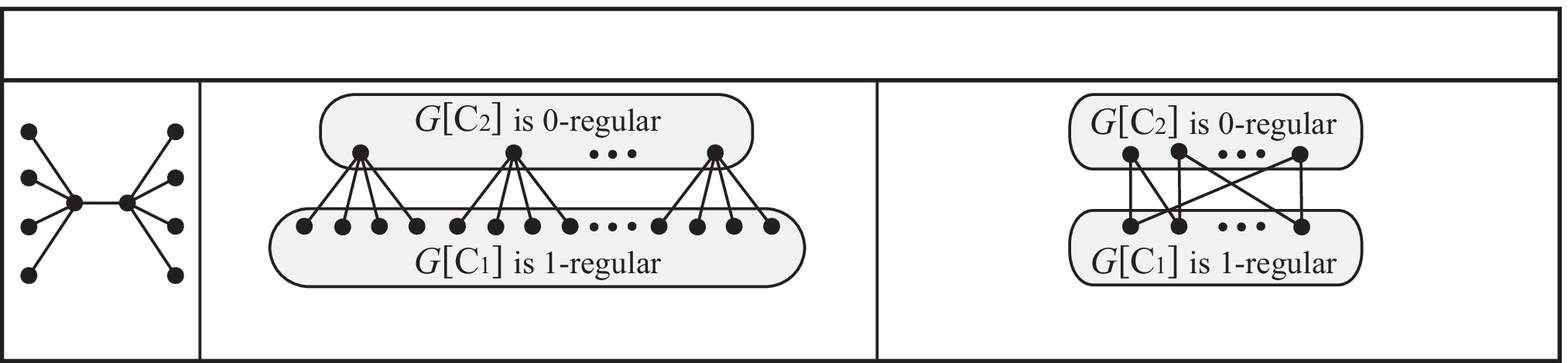}}}

\put(60,16){\small $S_{4,4}$}

\put(92,16){\small We collect such graph $G$ in $\mathcal {G}_1(1,4)\!$}
\put(251,16){\small We collect such graph $G$ in $\mathcal {G}_2(1,4)$}

\put(125,81){\small $\mathscr{G}(1,4)=\mathscr{G}_1(1,4)=\{S_{4,4}\}\cup \mathcal {G}_1(1,4)\cup \mathcal {G}_2(1,4)$}

\end{picture}
\end{center}
\vspace{-1.0cm} \caption{All the graphs in $\mathscr{G}(1,4)$ }\label{fig-6} \vspace{-0.2cm}
\end{figure}

\begin{thm}\label{E-35}
$\mathscr{G}(1,3)=\mathscr{G}_1(1,3)$, $\mathscr{G}(1,4)=\mathscr{G}_1(1,4)$   and $\mathscr{G}(1,5)=\mathscr{G}_1(1,5)\cup \mathscr{G}_2(1,5)$.
\end{thm}

\begin{proof}
First we know that $\mathscr{G}_1(1,b)\subseteq
\mathscr{G}(1,b)$ and $\mathscr{G}_2(1,5)\subseteq \mathscr{G}(1,5)$. Let $G\in \mathscr{G}(1,b)$ and $b\in \{3,4,5\}$. Then
$\delta(G)=1$ or $2$ by Lemma \ref{E-29}. Suppose that $\delta(G)=1$. Then
$G=S_{b,b}\in \mathscr{G}_1(1,b)$ by Lemma \ref{E-30}.
 Suppose that $\delta(G)=2$. Then $1+\frac{b}{2}\leq \Delta(G)\leq b$  by Lemma \ref{E-33}.
 If $\Delta(G)=b$,  then $G\in \mathcal {G}_1(1,b)\subseteq \mathscr{G}_1(1,b)$ again by
Lemma \ref{E-33}. If $1+\frac{b}{2}\leq \Delta(G)<b$, then $b\neq 3$ and \begin{math}
 \left\{
  \begin{array}{ccc}
      b&\!\!\!=\!\!\!&4\\
    \Delta(G)&\!\!\!=\!\!\!&3\\
  \end{array}\right.
\end{math}
or  \begin{math}
 \left\{
  \begin{array}{ccc}
      b&\!\!\!=\!\!\!&5\\
    \Delta(G)&\!\!\!=\!\!\!&4\\
  \end{array}.\right.
\end{math}
Therefore, $G=S_{3,3}$ or $G\in \mathcal {G}_1(1,3)$ if $b=3$, that is
$\mathscr{G}(1,3)\subseteq \mathscr{G}_1(1,3)$, and so $\mathscr{G}(1,3)=\mathscr{G}_1(1,3)$(see Figure \ref{fig-5}).

If \begin{math}
 \left\{
  \begin{array}{ccc}
      b&\!\!\!=\!\!\!&4\\
    \Delta(G)&\!\!\!=\!\!\!&3\\
  \end{array}\right.
\end{math}, note that $\delta(G)=2$, let
$u,v\in V(G)$, $N(u)=\{x_1,x_2\}$ and $N(v)=\{y_1,y_2,y_3\}$.
Applying (\ref{equ-1}) to $u$, we have $d(x_1)+d(x_2)=6$, which
implies that $(d(x_1),d(x_2))=(3,3)$. Applying (\ref{equ-1}) to $v$,
we have $d(y_1)+d(y_2)+d(y_3)=7$, which implies that
$(d(y_1),d(y_2),d(y_3))=(2,2,3)$. Let
 $$
\begin{array}{ll}C_1=\{u\in V(G)| d(u)=3 \}, \\
C_2=\{u\in V(G)|d(u)=2 \}.\\
 \end{array}
$$
Then $\pi:V(G)=C_1\cup C_2$ is a partition of $V(G)$. From the above
arguments we know that $|N(u)\cap C_1|=1$ and $|N(u)\cap C_2|=2$ for
any $u\in C_1$, $|N(u)\cap C_1|=2$ and $|N(u)\cap C_2|=0$   for any
$u\in C_2$. Therefore, $\pi:V(G)=C_1\cup C_2$ is a $2$-equitable partition
of $V(G)$  with  parameters $(1,2;2,0)$.  By the definition,  $G\in \mathcal {G}_2(1,4)\subseteq\mathscr{G}_1(1,4)$.
Combining the arguments in the first  paragraph, we know that  $G=S_{4,4}$ or $G\in \mathcal {G}_1(1,4)$ or $G\in \mathcal {G}_2(1,4)$ if $b=4$, that is $\mathscr{G}(1,4)\subseteq \mathscr{G}_1(1,4)$, and so $\mathscr{G}(1,4)=\mathscr{G}_1(1,4)$(see Figure \ref{fig-6}).

If \begin{math}
 \left\{
  \begin{array}{ccc}
      b&\!\!\!=\!\!\!&5\\
    \Delta(G)&\!\!\!=\!\!\!&4\\
  \end{array}\right.
\end{math},  note that $\delta(G)=2$, let $u,v\in V(G)$,
$N(u)=\{x_1,x_2\}$ and $N(v)=\{y_1,y_2,y_3,y_4\}$.
 Applying (\ref{equ-1}) to $u$, we have $d(x_1)+d(x_2)=7$.
 Since $\Delta(G)=4$, we have $(d(x_1),d(x_2))=(3,4)$.
 Applying (\ref{equ-1}) to $v$, we have $d(y_1)+d(y_2)+d(y_3)+d(y_4)=9$. Since $\delta(G)=2$,
we have $(d(y_1),d(y_2),d(y_3),d(y_4))=(2,2,2,3)$. Obviously, there
is a
  vertex $w\in V(G)$ with $d(w)=3$. Let $N(w)=\{z_1,z_2,z_3\}$. Applying (\ref{equ-1})
  to $w$,  we have
$d(z_1)+d(z_2)+d(z_3)=8$. Since $\delta(G)=2$, we have
$(d(z_1),d(z_2),d(z_3))=(2,2,4)$ or $(2,3,3)$. Let
\begin{equation}\label{equ-15}
\begin{array}{ll}V_1=\{v\in V(G)|d(v)=4,\mbox{and the degrees of the neighbors of $v$ are 2,2,2 and 3} \}, \\
V_2=\{v\in V(G)|d(v)=3,\mbox{and the degrees of the neighbors of $v$ are 2,2  and 4}\},\\
V_3=\{v\in V(G)|d(v)=3, \mbox{and the degrees of the neighbors of
$v$ are
2,3 and 3}\},\\
V_4=\{v\in V(G)|d(v)=2, \mbox{and the degrees of the neighbors of
$v$ are 3
and 4} \}.\\
 \end{array}
\end{equation}
Then $V_1\cup V_2\cup V_3 \cup V_4$ is a partition of $V(G)$. For convenience, if
$|N(v)\cap V_j|$ is a constant for any $v\in V_i$, then set
$r_{ij}=|N(v)\cap V_j|$. Note that  if $r_{ij}=0$ then $r_{ji}=0$,
and that $\sum_{j=1}^{4}|N(v)\cap V_j|=d(v)$ for any $v\in V(G)$,
combining (\ref{equ-15}) we have
$$
\left\{
\begin{array}{llll}r_{11}=0,& r_{12}=1,&r_{13}=0,&r_{14}=3,\\
r_{21}=1,&r_{22}=0,&r_{23}=0,&r_{24}=2,\\
r_{31}=0,&r_{32}=0,&r_{33}=2,&r_{34}=1,\\
r_{41}=1,&&&r_{44}=0.
\end{array}
\right.$$ Since $r_{ii}=0 (i=1,2,4)$, $G[V_i]$ is a vertex
independent set for $i=1,2,4$. Since $r_{33}=2$,  $G[V_3]$ is a
 $2$-regular subgraph. Let  $|V_1|=t$.  Since
$|V_1|r_{12}=|V_2|r_{21}$,
 we have $|V_2|=t$. Since
$|V_1|r_{14}=|V_4|r_{41}$,  we have $|V_4|=3t$.
  Note that $ |N(v)\cap V_2|+ |N(v)\cap V_3|=1$ for any $v\in V_4$,
  we have $|V_4|=|V_2|r_{24}+|V_3|r_{34}$, which gives $|V_3|=t$.
Since $G[V_3]$ is a
 $2$-regular subgraph, we have
$t\geq 3$.

Let $V_1=\{v_1^{(1)},v_2^{(1)},...,v_t^{(1)}\}$,
$V_2=\{v_1^{(2)},v_2^{(2)},...,v_t^{(2)}\}$,
$V_3=\{v_1^{(3)},v_2^{(3)},...,v_t^{(3)}\}$ and    $V_4=\cup
_{i=1}^3V_{i4}'$, where
$V_{i4}'=\{u_{i1}',u_{i2}',...,u_{it}'\}(i=1,2,3)$. Since $r_{12}=1$
and $r_{21}=1$, without loss of generality,  let the edges of $G$
between $V_1$ and $V_2$ are
$$\{v_i^{(1)}v_i^{(2)}\mid i=1,...,t\}.$$
Since $r_{41}=1$ and $r_{14}=3$, without loss of generality,  let
the edges of $G$ between $V_1$ and $V_4$ are
$$\{v_i^{(1)}u_{ji}'|i=1,2,...,t; j=1,2,3\}. $$
Note that  $r_{24}=2$, $r_{34}=1$ and
 $ |N(v)\cap V_2|+ |N(v)\cap V_3|=1$ for any $v\in V_4$, we can relabeling the vertices in
$V_4$  such that $V_4=\cup _{i=1}^3V_{i4}$, where
$V_{i4}=\{u_{i1},u_{i2},...,u_{it}\}(i=1,2,3)$, and the other edges
of $G$ are
$$ \left\{\begin{array}{ll}
\mbox{ the edges between $V_2$ and $V_4$: $\{ v_i^{(2)}u_{1i},  v_i^{(2)}u_{2i}\mid i=1,...,t\}$},\\
\mbox{ the edges between $V_3$ and $V_4$: $\{ v_i^{(3)}u_{3i}\mid i=1,...,t\}$}.\\
\end{array}\right.
$$
  By the definition,
$G\in \mathscr{G}_2(1,5)$.
Combining the arguments in the first  paragraph, we know that $G=S_{5,5}$ or $G\in \mathcal {G}_1(1,5)$ or $G\in \mathscr{G}_2(1,5)$ if $b=5$,  that is $\mathscr{G}(1,5)\subseteq \mathscr{G}_1(1,5)\cup \mathscr{G}_2(1,5)$, and so $\mathscr{G}(1,5)=\mathscr{G}_1(1,5)\cup \mathscr{G}_2(1,5)$(see Figure \ref{fig-7}).
\end{proof}

\begin{figure}
\begin{center}
\begin{picture}(550,136)
\put(36,10){\resizebox{11cm}{!}{\includegraphics[0,0][500,100]{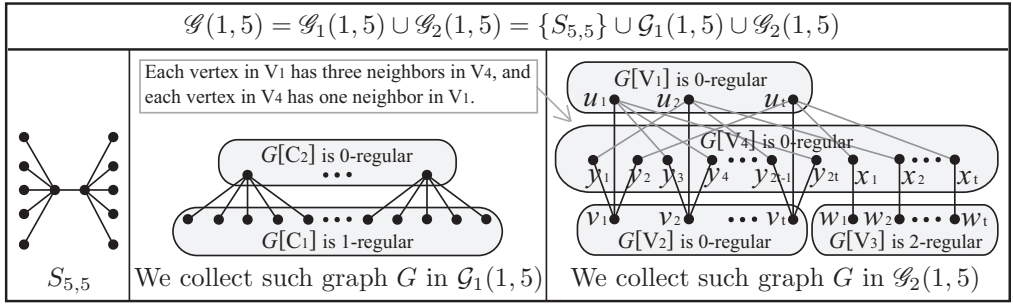}}}

\put(52,16){\small $S_{5,5}$}

\put(85,16){\small We collect such graph $G$ in $\mathcal {G}_1(1,5)\!$}
\put(250,16){\small We collect such graph $G$ in $\mathscr{G}_2(1,5)$}

\put(104,112){\small $\mathscr{G}(1,5)=\mathscr{G}_1(1,5)\cup\mathscr{G}_2(1,5)=\{S_{5,5}\}\cup \mathcal {G}_1(1,5)\cup \mathscr{G}_2(1,5)$}

\end{picture}
\end{center}
\vspace{-1.0cm} \caption{All the graphs in $\mathscr{G}(1,5)$ }\label{fig-7} \vspace{-0.2cm}
\end{figure}

At the last of this section we propose a problem to characterize the graphs in
$\mathscr{G}(a,b)$ for other feasible pairs $(a,b)$.


{\footnotesize \begin{thebibliography}{11}



\bibitem{Cve}Cvetkovi\'{c} D.M.,  The main part of the spectrum, divisors and switching of graphs,  Publ. Inst. Math. (Beograd) (N.S.),  23(37), 31--38,  (1978).

\bibitem{Rowlinson}Rowlinson P., The main eigenvalues of a graph: a survey, Appl. Analysis and Discete Math, 1, 445--471, (2007).

\bibitem{Hagos}Hagos E.M., Some results on graph spectra, Linear Algebra Appl.,
356, 103--111, (2002).


\bibitem{Hou1}Hou Y.P., Tian F.,Unicyclic graphs with exactly two main
eigenvalues, Appl. Math. Lett., 19, 1143--1147, (2006).

\bibitem{Teranishi} Teranishi Y., Main eigenvalues of a graph, Linear and Multilinear
Algebra, 49(4), 289--303, (2001).


\bibitem{Hou2}Hou Y.P., Zhou H.Q., Trees with exactly two mian
eigenvalues, Acta of Hunan Normal University, 28(2), 1--3, (2005) (in
Chinese).



\bibitem{Shi}Shi L.S., On graphs with given main eigenvalues, Appl. Math. Lett., 22, 1870--1874, (2009).



\bibitem{Hu}Hu Z.Q., Li S.C., Zhu C.F., Bicyclic graphs with exactly two main
eigenvalues,  Linear Algebra Appl.,  431, 1848--1857, (2009).



\bibitem{Hou3}Hou Y.P., Tang Z.K., Shiu W.C.,  Some results on graphs with exactly two main eigenvalues,
 Appl. Math. Lett.,  25, 1274--1278, (2012).



\bibitem{Tang1}Tang Z.K.,  Hou Y.P.,  Tricyclic Graphs with Exactly Two Main Eigenvalues,
Acta of Hunan Normal University, 34(4), 7--12, (2011) (in chinese).



\bibitem{Fan}Fan X.X., Luo Y.F., Gao X.,  Tricyclic graphs with exactly two main eigenvalues,
Central European Journal of Mathematics, 11(10), 1800--1816, (2013).


\bibitem{Tang2} Tang Z.K., Hou Y.P., The integral graphs with index
$3$ and exactly two main eigenvalues, Linear Algebra Appl., 433,
984--993, (2010).





\end{thebibliography}}
\end{document}